\documentclass[12pt]{article}
\usepackage{amsmath, amsfonts,amsthm,amssymb,amsbsy,upref,color,graphicx,amscd,hyperref,makeidx,enumerate}
\usepackage{tikz}
\usepackage[active]{srcltx}
\usepackage[latin1]{inputenc}
\usepackage{enumerate}
\usepackage[english]{babel}
\usepackage{newlfont}
\usepackage{amsbsy}
\usepackage[english]{babel}
\usepackage[center]{caption2}
\usepackage{amssymb}
\usepackage{amsmath}
\usepackage{latexsym}
\usepackage{amsthm}
\usepackage{amsthm}
\theoremstyle{plain}
\newtheorem{thm}{Theorem}

\newtheorem{lem}[thm]{Lemma}
\newtheorem{prop}[thm]{Proposition}
\newtheorem{nota}[thm]{Notation}
\newtheorem{rem}[thm]{Remark}
\newtheorem{defin}[thm]{Definition}

\newcommand{\R}{\mathbb{R}}

\newcommand{\Z}{\mathbb{Z}}
\newcommand{\N}{\mathbb{N}}

\usepackage{mathtools}
\def\multiset#1#2{\ensuremath{\left(\kern-.2em\left(\genfrac{}{}{0pt}{}{#1}{#2}\right)\kern-.2em\right)}}
\usepackage[center]{caption2}

\begin{document}

\title{Treelike families of multiweights}
\author{Agnese Baldisserri \  \ \  \ \ Elena Rubei}
\date{}
\maketitle

\def\thefootnote{}
\footnotetext{ \hspace*{-0.36cm}

 }

\begin{abstract}
Let  ${\cal T}=(T,w)$ be a weighted  finite tree with leaves $1,..., n$.
For any  $I :=\{i_1,..., i_k \} \subset \{1,...,n\}$,
let $D_I ({\cal T})$ be
the weight of the minimal subtree of $T$ connecting $i_1,..., i_k$;
the  $D_{I} ({\cal T})$ are called 
$k$-weights of  ${\cal T}$. Given  a family 
of  real numbers parametrized by the $k$-subsets of $
\{1,..., n\}$, $\{D_I\}_{I  \in {\{1,...,n\} \choose k}}$,
  we say that a weighted tree ${\cal T}=(T,w)$ with leaves $1,..., n$ realizes the family 
if $D_I({\cal T})=D_I$ for any $ I  $.
  We give a characterization of the families of real numbers that are realized by some weighted tree.
\end{abstract}

\section{Introduction}

For any graph $G$, let $E(G)$, $V(G)$ and $L(G)$ 
 be respectively the set of the edges,   
the set of the vertices and  the set of the leaves of $G$.
A {\bf weighted graph} ${\cal G}=(G,w)$ is a graph $G$ 
endowed with a function $w: E(G) \rightarrow \R$. 
For any edge $e$, the real number $w(e)$ is called the weight of the edge. If all the weights are nonnegative (respectively positive), 
we say that the graph is {\bf nonnegative-weighted} (respectively {\bf positive-weighted});
if  the weights  of 
the internal edges are nonzero  (respectively positive),
 we say that the graph is {\bf internal-nonzero-weighted}
 (respectively {\bf internal-positive-weighted}).
For any finite subgraph $G'$ of  $G$, we define  $w(G')$ to be the sum of the weights of the edges of $G'$. 
In this paper we will deal only with weighted finite trees.

\begin{defin}
Let ${\cal T}=(T,w) $ be a weighted tree. 
For any distinct $ i_1, .....,i_k \in V(T)$,
 we define $ D_{\{i_1,...., i_k\}}({\cal T}) $ to be the weight of the minimal 
subtree containing $i_1,....,i_k$. We call this subtree ``the subtree realizing  $ D_{\{i_1,...., i_k\}}({\cal T}) $''.
 More simply, we denote 
$D_{\{i_1,...., i_k\}}({\cal T})$ by
$D_{i_1,...., i_k}({\cal T})$ for any order of $i_1,..., i_k$.  
We call  the  $ D_{i_1,...., i_k}({\cal T})$ 
the {\bf $k$-weights} of ${\cal T}$ and
 we call  a $k$-weight of ${\cal T}$ for some $k$  a {\bf multiweight}
 of ${\cal T}$.
\end{defin}

%Observe that in the case ${\cal G}$ is a tree,  $ D_{i_1,...., i_k}({\cal G})$ is the sum of the weights of the edges of the  minimal subtree joining $i_1,...,i_k$.

For any $S \subset V(T)$,
we call  the family $\{D_I\}_{I  \in {S \choose k}}$
the {\bf family of the $k$-weights}  of $({\cal T}, S)$ or the $k$-dissimilarity family of   $({\cal T}, S)$.

We can wonder when a family of real numbers is the family of the $k$-weights
of some weighted tree and of some subset of the set of its vertices.
If $S$ is a finite set, $k \in \N$ and $k < \# S$,
we say that a family of  real numbers
 $\{D_{I}\}_{I \in {S \choose k}}$  is {\bf treelike} (respectively p-treelike, nn-treelike, 
inz-treelike, ip-treelike, nn-ip-treelike) if there exist a weighted  (respectively 
positive-weighted,  nonnegative-weighted, internal-nonzero-weighted, 
internal-positive-weighted, nonegative-weighted and internal-positive weighted)  tree
 ${\cal T}=(T,w)$ and a subset $S$ 
of the set of its vertices such that $ D_{I}({\cal T}) = D_{I}$  for any 
 $k$-subset $I$ of $ S$.
In this case, we say that ${\cal T}=(T,w)$ {\bf realizes the family $\{D_I\}_{I \in {S \choose k} }$}.

If in addition 
% the tree is a weighted (respectively  positive-weighted, nonnegative-
%weighted, internal-nonzero-weighted, internal-positive-weighted) tree and 
$S \subset L(T)$,
we say that the family
 is  {\bf l-treelike} (respectively p-l-treelike, nn-l-treelike, inz-l-treelike,
  ip-l-treelike,   nn-ip-l-treelike).

A criterion for a metric on a finite set to be nn-l-treelike
%realized by a  nonnegative-weighted tree with leaf set $\{1,..., n\}$ 
was established in  \cite{B}, \cite{SimP}, \cite{Za}:

\begin{thm} \label{Bune} %{\bf (Buneman)} 
Let $\{D_{I}\}_{I \in {\{1,...,n\} \choose 2}}$ be a set of positive real numbers 
satisfying the triangle inequalities.
It is p-treelike (or nn-l-treelike) 
 if and only if, for all distinct $i,j,k,h  \in \{1,...,n\}$,
the maximum of $$\{D_{i,j} + D_{k,h},D_{i,k} + D_{j,h},D_{i,h} + D_{k,j}
 \}$$ is attained at least twice. 
\end{thm}

%In terms of tropical geometry, the theorem above can be formulated %by saying that the set of the $2$-dissimilarity vectors of weighted %trees with $n$ leaves and such that  the internal 
%edges have  negative weights is the tropical Grassmannian $ {\cal %G}_{2,n}$ (see \cite{SS2}).

In \cite{B-S}, Bandelt and Steel proved a result, analogous to
Theorem \ref{Bune}, for general weighted trees; more precisely, they
 proved that, for any set of real numbers $\{D_{I}\}_{I \in {\{1,...,n\} \choose 2}}$,
 there exists a weighted tree ${\cal T}$ with leaves $1,...,n$
such that $ D_{I} ({\cal T})= D_{I}$  for any $I \in {\{1,...,n\} \choose 2}$  if and only  if, for any $a,b,c,d \in  \{1,...,n\}$, at least two among 
 $ D_{a,b} + D_{c,d},\;\;D_{a,c} + D_{b,d},\;\; D_{a,d} + D_{b,c}$
are equal.

For higher $k$ the literature is more recent. Some of the most important results are the following.

\begin{thm} \label{PatSp} {\bf (Pachter-Speyer,  \cite{P-S})}. Let $ k ,n  \in \mathbb{N}$ with
$3 \leq  k \leq (n+1)/2$.  A positive-weighted tree
 ${\cal T}$ with leaves $1,...,n$ and no vertices of degree 2
is determined by the values $D_I({\cal T})$, where $ I $ varies in  
${\{1,...,n\} \choose k }$.
\end{thm}

%In \cite{BC}  Bocci and Cools gave a characterization of $3$-
%dissimilarity vectors of  trees with $n$ leaves in term of the tropical %Grassmannian ${\cal G}_{3,n}$ and in \cite{Iri} and \cite{Man} Iriarte %Giraldo and Manon proved that $k$-dissimilarity vectors of  trees with %$n$ leaves are contained in the tropical Grassmannian ${\cal %G}_{k,n}$. 

\begin{thm} {\bf  (Herrmann, Huber, Moulton, Spillner, \cite{H-H-M-S})}.
If $n \geq 2k$, a family  of positive real numbers 
$ \{D_{I}\}_{I \in {\{1,..., n\} \choose k}}$ is nn-ip-l-treelike 
if and only if the restriction to every $2k$-subset of $\{1,...,n\}$ is 
nn-ip-l-treelike.   
\end{thm}

%Furthermore, they studied when a family  of positive real numbers %indexed by ${\{1,...,n\} \choose k}$ is nn-ip-l-treelike in the case $k=3$.
%We quote also: 

 \begin{thm} {\bf (Levy-Yoshida-Pachter)} Let ${\cal T}=(T,w)$ be a positive-weighted tree 
  with $L(T)=\{1,\ldots,n\}$. For any $i ,j \in \{1,\ldots,n\}$, define $$S(i,j) = \sum_{Y \in {\{1,\ldots,n\} -\{i,j\} \choose k-2}} D_{i,j ,Y} ({\cal T}). $$ Then there exists a positive-weighted tree ${\cal T}' =(T',w')$   such that $D_{i,j}({\cal T}')=S(i,j)$  for all $i,j \in \{1,\ldots,n\}$,
  the quartet system of $T'$ contains  the quartet system of $T$ and, defined $T_{\leq s}$ the subforest of $T$  whose edge set consists of edges whose removal results in one of the components having size at most $s$, we have 
  $T_{\leq n-k} \cong T'_{\leq n-k}$. 
  \end{thm}

In \cite{Ru1} and \cite{Ru2},  the author 
gave an inductive characterization of the families of real numbers 
that are indexed by the subsets of $\{1,...,n\}$
 of  cardinality greater than or equal to $2$ and  
 are the families of  the multiweights of a tree with $n$ leaves.

Let $n ,k \in \N$ with $n >k$.
In  \cite{B-R}    we  studied the problem of the  characterization
of  the  families of positive real numbers, indexed by the $k$-subsets
 of an $n$-set, that are p-treelike in the ``border'' 
case $k=n-1$. 
Moreover we studied the analogous problem  for graphs.
See \cite{H-Y}  for other results on graphs and see the introduction of \cite{B-R} for a survey.

Here we examine  the case of  trees for general $k$. To illustrate the result, we need the following definition.

\begin{defin} Let $k\in \N- \{0\}$.  We say that  a tree $P$ is a {\bf pseudostar} of kind $(n,k)$ if $\# L(P)=n$ and
   any edge  of $P$ divides $L(P)$
 into two sets such that  at least one of them has cardinality   greater than or equal to $ k$.
\end{defin}

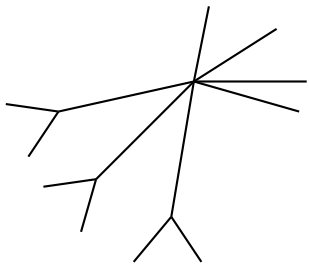
\begin{figure}[h!]
\begin{center}

\begin{tikzpicture}
\draw [thick] (0,0) --(0.2,1);
\draw [thick] (0,0) --(1.1,0.7);
\draw [thick] (0,0) --(1.5,0);
\draw [thick] (0,0) --(1.4,-0.4);
\draw [thick] (0,0) --(-1.8,-0.4);
\draw [thick] (-1.8,-0.4) --(-2.5,-0.3);
\draw [thick] (-1.8,-0.4) --(-2.2,-1);
\draw [thick] (0,0) --(-1.3,-1.3);
\draw [thick] (-1.3,-1.3) --(-2,-1.4);
\draw [thick] (-1.3,-1.3) --(-1.5,-2);
\draw [thick] (0,0) --(-0.3,-1.8);
\draw [thick] (-0.3,-1.8) --(-0.8,-2.4);
\draw [thick] (-0.3,-1.8) --(0.1,-2.4);

\end{tikzpicture}

\caption{ A pseudostar of kind $(10,8)$  \label{pseudostar}}
\end{center}
\end{figure}

In \cite{B-R2}
we proved that, if $3 \leq k \leq n-1$,  given a l-treelike family of  real numbers,  $\{ D_I\}_{ I  \in {\{1,...,n\} \choose k }}$, 
there exists exactly one internal-nonzero-weighted pseudostar ${\cal P}$ of kind $(n,k)$  with leaves $1,...,n$ and no vertices of degree 2  
such that  $D_I({\cal P})=D_I$ for any $ I  $. 
Here we associate to any  pseudostar of kind $(n,k)$ with  leaves $1,..., n$  
 a hierarchy on $\{1,...,n\}$ with clusters of cardinality between $2$ and $n-k$ and,
 by using this association and 
 by pushing forward the ideas in \cite{Ru1} and \cite{Ru2}, 
 we get a theorem  (Theorem \ref{char}) characterizing  l-treelike dissimilarity families;
consequently, we obtain also  a characterization of  p-l-treelike dissimilarity families (see Remark \ref{charpos}).

%???????????Finally we want to mention that, 
%in \cite{Iri} and \cite{Man}, the authors proved that the $k$-dissimilarity vectors 
%of trees are contained in the tropical Grassmannian for any $k$, so %the equations of tropical Grassmannian are necessary  for a %dissimilarity vector to be p-l-treelike, but they do not seem sufficient.
%Our conditions in Theorem \ref{main}, which are necessary and %sufficient conditions for an 
%$k$-dissimilarity vector to be p-l-treelike, imply the equations of %tropical Grassmannians. So in a certain sense, we can say that we %have found out why the
%equations of tropical Grassmannian are not sufficient to characterize %the  p-l-treelike dissimilarity vectors.

%%%%%%%%%%%%%%%%%%%%%%%%%%%%%%%%%%%%%%%%%%%
%%%%%%%%%%%%%%%%%%%%%%%%%%%%%%%%%%%%%%%%%%%

\section{Notation and recalls}

\begin{nota} \label{notainiziali}

%$\bullet$  Let $ \mathbb{R}_{+} =\{x \in \mathbb{R} | \; x >0\}$.

$\bullet$ We use the symbols $\subset$ and $\subsetneq$  respectively for the inclusion and the strict inclusion.

$\bullet$  For any $n \in \N $ with $ n \geq 1$, let $[n]= \{1,..., n\}$.

$\bullet$  For any set $S$ and $k \in \mathbb{N}$,  let ${S \choose k}$
be the set of the $k$-subsets of $S$ and  let ${S \choose \geq k}$
be the set of the $t$-subsets of $S$ with $t \geq k$.

$\bullet$
For any  $A,B \subset [n]$, we will write $AB$ instead of $A \cup B$. Moreover, we will write $a,B$, or even $aB$, instead of $\{a\} \cup B$.

$\bullet$ Throughout the paper, the word ``tree'' will denote a finite  tree.

$\bullet$  We say that a vertex of a tree is a {\bf node} if its degree is greater than $2$.
  
$\bullet$  Let $F$ be a leaf of a tree $T$. Let $N$ be the node 
 such that the path $p$ between $N$ and $F$ does not contain any node apart from $N$. We say that $p$ is the {\bf twig} associated to $F$.
We say that an edge is {\bf internal} if it is not an edge of a twig.

$\bullet $ We say that a tree is {\bf essential} if it has no vertices of degree $2$.

$\bullet$ Let $T$ be a tree and let $S$ be a subset of $L(T)$. We denote by $T|_S$ the minimal subtree 
of $T$ whose set of vertices  contains $S$. If ${\cal T}= (T,w)$ is a weighted tree, we denote by 
${\cal T}|_S$  the tree $T|_S$ with the weight induced by $w$.

\end{nota}

\begin{defin} Let $T$ be a tree.

We say that two leaves  $i$ and $j$ of $T$ are {\bf neighbours}
if in the path from $i$ to $j$ there is only one node; 
furthermore, we say that  $C \subset L(T)$ is a {\bf cherry} if any $i,j \in C$ are neighbours.

We say that a cherry  is {\bf complete}
 if it is not strictly  contained in another cherry.

The {\bf stalk} of a cherry is the unique node in the path 
with endpoints any two elements of the cherry.

Let $C$ be a cherry in $T$.
We say that a tree $T'$ is obtained from $T$ by {\bf pruning} $C$ if it is obtained from
$T$ by ``deleting'' all the twigs associated to   leaves of $C$ (more precisely, by contracting all the edges of the twigs associated to leaves of $C$).

We say that a cherry $C$ in $T$ is {\bf good} if it is complete and, if $T'$ is the tree obtained from $T$ by pruning $C$, the stalk of $C$ is a leaf of $T'$.
We say that  a cherry is {\bf bad} if it is not good.

Let  $i,j,l,m \in L(T)$. We say that $ \langle
i, j | l, m \rangle $ holds if  in  $T|_{\{i,j,l,m\}}$ we have that $i$ and $j$ are neighbours, 
 $l$ and $m$ are neighbours, and  $i$ and $l$ are not neighbours; in this case we denote by  $\gamma_{i,j,l,m}$ the path 
between the stalk of $\{i,j\}$ and the stalk of $ \{l,m\}$ in $T|_{\{i,j,l,m\}}$. The symbol 
$ \langle i, j | l, m \rangle $  is called {\bf Buneman's index} of $i,j,l,m$. 

\end{defin}

{\bf Example.} In the tree in Figure \ref{good} the only good cherries are $\{1,2,3\}$ and $\{6,7\}$.

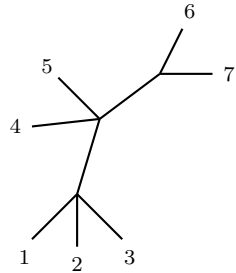
\begin{figure}[h!]
\begin{center}

\begin{tikzpicture}
\draw [thick] (0,0) --(0.8,0.6);
\draw [thick] (0.8,0.6)--(1.1,1.2);
\draw [thick] (0.8,0.6)--(1.5,0.6);
\draw [thick] (0,0) --(-0.3,-1);
\draw [thick] (0,0) --(-0.9,-0.1);
\draw [thick] (0,0) --(-0.55,0.55);
\draw [thick] (-0.3,-1)--(-0.9, -1.6);
\draw [thick] (-0.3,-1)--(-0.3, -1.7);
\draw [thick] (-0.3,-1)--(0.3, -1.6);
\node[above] at (1.2,1.2) {\scriptsize 6};
\node[right] at (1.5,0.6) {\scriptsize 7};
\node[left] at (-0.9,-0.1) {\scriptsize 4};
\node at (-0.7,0.7) {\scriptsize 5};
\node[below] at (-1, -1.6) {\scriptsize 1};
\node[below] at (-0.3, -1.7) {\scriptsize 2};
\node[below] at (0.4, -1.6) {\scriptsize 3};
\end{tikzpicture}

\caption{Good cherries and bad cherries} \label{good}
\end{center}
\end{figure}

\begin{rem}
(i) A pseudostar of kind $(n,n-1) $ is a star, that is, a tree with only one node.

%(ii)  Let $ r \in \N -\{0\}$. In any pseudostar of kind $(n, n-r)$ that is not a %star, all the complete cherries, except  at most one, have cardinality  %less than or equal to $r$.

%(iii) Let $ r \in \N-\{0\}$. By pruning  a cherry in a $r$-pseudostar, we get again a $r$-%pseudostar.

(ii) Let $ k,n \in \N-\{0\}$. If $\frac{n}{2} \geq k $, then every tree 
with $n$ leaves is a pseudostar of kind $(n,k)$, in fact if we divide a set with $n$ elements into two 
parts, at least one of them has cardinality greater than or equal to $ \frac{n}{2}$, which 
is greater than or equal to $ k$. 
\end{rem}

\begin{defin} Let $S$ be  a set. We say that a set system ${\cal H}$ of $S$ is a {\bf hierarchy} over $S$
if, for any $ H, H' \in {\cal H}$, we have that $ H \cap H' $ is one among $ \emptyset, H, H'$. We say that ${\cal H}$ {\bf covers} $S$ if 
$S= \cup_{H \in {\cal H}} H$.
\end{defin}

\begin{defin} \label{hier} Let $k,n \in \N$ with $ 2 \leq k \leq n-2$.
Let $P$ be an essential pseudostar of kind $(n,k)$ with  $L(P)=[n]$. 

We define inductively $P^s$ for any $ s \geq 0 $ as follows: let $P^0=P$ 
and let $P^s$ be 
the tree obtained from $P^{s-1}$ by pruning all the good cherries of cardinality $ \leq n-k$.

We say that $x \in [n]$ descends from $y \in L(P^{s})$ for some $s$ if the path between $x$ and $y$ in $P$ 
contains no leaf of $P^s$ apart from $y$. For any $Y \subset L(P^{s})$, let $ \partial Y $ denote the 
subset of the elements of $[n]$ descending from any element of $Y$.

We  define   {\bf the hierarchy   ${\cal H}$  over $[n]$ associated to the pseudostar $P$}   (depending on $k$) as follows:

we say that a cherry $C$ of $P$ is in ${\cal H}$ if and only if $C$ is good and  $\# C \leq n-k$;

if $C$ is a cherry of $P^{s}$ for some $s$,
 we say that $ \partial C $ is in ${\cal H}$ if and only if $C$ is good and
 $\# \partial C \leq n-k$; 

if, for some $s$, we have that $L(P^{s})$ is the union of
two complete cherries, $C_1$ and $C_2$, and  both have cardinality less than or equal to $n-k$, we put in ${\cal H}$ only $\partial C_i$ for $i$ such that $\partial C_i$ contains the minimum of $\partial C_1 \cup \partial C_2$.

The elements of ${\cal H}$ are only the ones above. We call the elements of ${\cal H}$ ``${\cal H}$-clusters''.

For any $H \in {\cal H}$, we define $e_H $  as follows: let $H= \partial C$ for  some $C$ cherry of $P^{s}$;
we call $e_H$ the twig of $P^{s+1}$  associated to the stalk of $C$.
For any $ i \in [n]$, we call $e_i$ the twig associated to $i$. 

\end{defin}

Observe that  the set of the leaves of a  star is a bad cherry;  so, 
according to our definition of ${\cal H}$,
if for some $s \in \N$ we have that $P^s$ is a star,   we do not consider $\partial L(P^s) $, which is $[n]$, a cluster of  ${\cal H}$. 
So, for instance, the hierarchy of a star is empty.

\bigskip

{\bf Examples.} Let $P$ be the pseudostar of kind $(12,6)$ in Figure \ref{PH} (a). 
The associated  hierarchy over $[12]$ (with $k=6$) is
 $${\cal H}= \{  \{4,5,6\}, \{7,8,9\}, \{1,2,3,4,5,6\} \}. $$
%The edge $e$ is $e_{\{1,2,3,4,5,6\}}$ if we choose the first hierarchy, it is $e_{\{7,8,9,10,11,12\}}$ if we choose the second hierarchy.

Let $Q$ be the pseudostar  of kind $(10,5)$ in Figure \ref{PH} (b). 
The associated  hierarchy over $[10]$  (with $k=5$)  is
$${\cal H}= \{  \{1,2\}, \{3,4\}, \{1,2,3,4\}  \}.$$
Let $R$ be the pseudostar of kind $(12,5)$ in Figure \ref{PH} (c). 
The associated  hierarchy over $[12]$ is
 $${\cal H}= \{  \{3,4,5\}, \{6,7\}, \{8,9\},\{1,2,3,4,5\} \}.$$

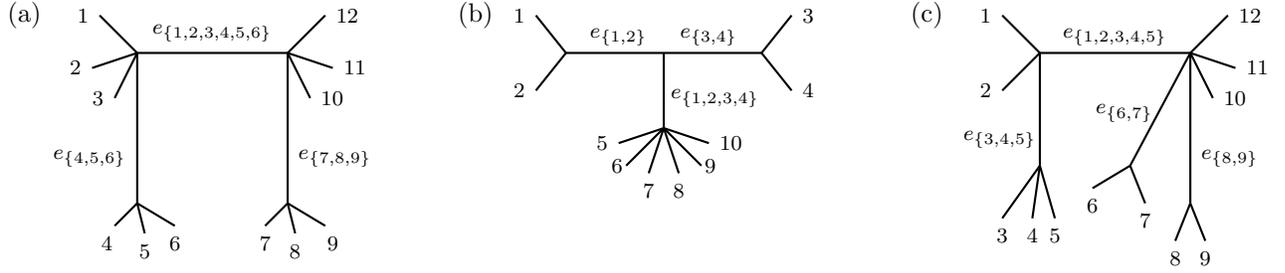
\begin{figure}[h!]

\begin{center}

\begin{tikzpicture}

\draw [thick] (-1.3,0) --(1.3,0) ; 
\draw [thick] (-1.3,0) --(-1.7,0.5);
\draw [thick] (1.3,0) --(1.7,0.5);
\draw [thick] (-1.3,0) --(-1.7,-0.5);
\draw [thick] (1.3,0) --(1.7,-0.5); 
\draw [thick] (0,0) --(0,-1.);
\draw [thick] (0,-1) --(-0.6,-1.2);
\draw [thick] (0,-1) --(-0.5,-1.5);
\draw [thick] (0,-1) --(-0.2,-1.6);
\draw [thick] (0,-1) --(0.6,-1.2);
\draw [thick] (0,-1) --(0.5,-1.5);
\draw [thick] (0,-1) --(0.2,-1.6);

\node [left] at (-1.7,0.5) {\scriptsize 1};
\node [left] at (-1.7,-0.5) {\scriptsize 2};
\node [right] at (1.7,0.5) {\scriptsize 3};
\node [right] at (1.7,-0.5) {\scriptsize 4};
\node [left] at (-0.6,-1.2) {\scriptsize 5};
\node [left] at (-0.4,-1.5) {\scriptsize 6};
\node [below] at (-0.2,-1.6) {\scriptsize 7};
\node [right] at (0.6,-1.2) {\scriptsize 10};
\node [right] at (0.4,-1.5) {\scriptsize{9}};
\node [below] at (0.2,-1.6) {\scriptsize{8}};
\node[above] at (-0.6,-0.05) {\scriptsize $e_{\{1,2\}}$};
\node[above] at (0.6,-0.05) {\scriptsize $e_{\{3,4\}}$};
\node[right] at (-0.05,-0.6) {\scriptsize $e_{\{1,2,3,4\}}$};

\draw [thick] (-7,0) --(-5,0) ; 
\draw [thick] (-7,0) --(-7.5,0.5);
\draw [thick] (-7,0) --(-7.6,-0.2);
\draw [thick] (-7,0) --(-7.3,-0.6);
\draw [thick] (-5,0) --(-4.5,0.5);
\draw [thick] (-5,0) --(-4.4,-0.2); 
\draw [thick] (-5,0) --(-4.7,-0.6); 
\draw [thick] (-7,0) --(-7,-2) ; 
\draw [thick] (-5,0) --(-5,-2);
\draw [thick] (-5,-2) --(-5.3,-2.3);
\draw [thick] (-5,-2) --(-4.9,-2.4);
\draw [thick] (-5,-2) --(-4.5,-2.3); 
\draw [thick] (-7,-2) --(-7.3,-2.3);
\draw [thick] (-7,-2) --(-6.9,-2.4);
\draw [thick] (-7,-2) --(-6.5,-2.3);

\node [left] at (-7.5,0.5) {\scriptsize{1}};
\node [left] at (-7.6,-0.2) {\scriptsize{2}};
\node [left] at (-7.3,-0.6) {\scriptsize{3}};
\node [right] at (-4.5,0.5) {\scriptsize 12};
\node [right] at (-4.4,-0.2) {\scriptsize 11};
\node [right] at (-4.7,-0.6) {\scriptsize 10};
\node [below] at (-4.4,-2.3) {\scriptsize{9}};
\node [below] at (-5.3,-2.3) {\scriptsize{7}};
\node [below] at (-4.9,-2.4) {\scriptsize{8}};
\node [below] at (-6.5,-2.3) {\scriptsize{6}};
\node [below] at (-7.4,-2.3) {\scriptsize{4}};
\node [below] at (-6.9,-2.4) {\scriptsize{5}};
\node[right] at (-5,-1.4) {\scriptsize $e_{\{7,8,9\}}$};
\node[left] at (-7,-1.4) {\scriptsize $e_{\{4,5,6\}}$};
\node[above] at (-6,0) {\scriptsize $e_{\{1,2,3,4,5,6\}}$};

\draw [thick] (5,0) --(7,0) ; 
\draw [thick] (7,0) --(7.5,0.5);
\draw [thick] (7,0) --(7.6,-0.2);
\draw [thick] (7,0) --(7.3,-0.6);
\draw [thick] (5,0) --(4.5,0.5);
\draw [thick] (5,0) --(4.5,-0.5); 
\draw [thick] (7,0) --(7,-2) ; 
\draw [thick] (5,0) --(5,-1.5);
\draw [thick] (5,-1.5) --(5.2,-2.2);
\draw [thick] (5,-1.5) --(4.9,-2.2);
\draw [thick] (5,-1.5) --(4.5,-2.2); 
\draw [thick] (7,-2) --(7.2,-2.5);
\draw [thick] (7,-2) --(6.8,-2.5); 
\draw [thick] (7,0) --(6.2,-1.5) ;
\draw [thick] (6.2,-1.5) --(5.7,-1.8) ;
\draw [thick] (6.2,-1.5) --(6.4,-2) ;

\node [left] at (4.5,0.5) {\scriptsize{1}};
\node [left] at (4.5,-0.5) {\scriptsize{2}};
\node [below] at (5.2,-2.2) {\scriptsize{5}};
\node [below] at (4.9,-2.2) {\scriptsize 4};
\node [below] at (4.5,-2.2) {\scriptsize 3};
\node [right] at (7.5,0.5) {\scriptsize 12};
\node [right] at (7.6,-0.2) {\scriptsize{11}};
\node [right] at (7.3,-0.6) {\scriptsize{10}};
\node [below] at (7.2,-2.5) {\scriptsize{9}};
\node [below] at (6.8,-2.5) {\scriptsize{8}};
\node [below] at (5.7,-1.8) {\scriptsize{6}};
\node [below] at (6.4,-2) {\scriptsize{7}};
\node[left] at (5.1,-1.15) {\scriptsize $e_{\{3,4,5\}}$};
\node[left] at (6.65,-0.8) {\scriptsize $e_{\{6,7\}}$};
\node[right] at (7,-1.4) {\scriptsize $e_{\{8,9\}}$};
\node[above] at (6,-0.05) {\scriptsize $e_{\{1,2,3,4,5\}}$};

\node at (-8.5,0.5) {\footnotesize (a)};
\node at (-2.5,0.5) {\footnotesize(b)};
\node at (3.5,0.5) {\footnotesize(c)};
\end{tikzpicture}

\caption{Pseudostars and hierarchies \label{PH}}
\end{center}
\end{figure}

\begin{rem} \label{fromhiertops}
It is easy to reconstruct the pseudostar $P$ from the hierarchy ${\cal H}$:

Let ${\cal H}$ be a hierarchy on $[n]$ such that its clusters have cardinality between $2$ and $n-k$.
Let us consider a star $B$ with $L(B) = [n] - \cup_{H \in {\cal H}} H$ and call $O$ its stalk. 
For any $M$ maximal element of ${\cal H}$, we add an edge $e_M$ with endpoint  $O$;
let $V_M$ be the other endpoint of $e_M$. Then we add a cherry with stalk $V_M$ and leaves $ M- \cup_{H \in {\cal H}, H \subsetneq M} H$ ; for every element $M'$ of ${\cal H}$ strictly contained in $M$ which is maximal among the elements of ${\cal H}$ strictly contained in $M$,  
we add an edge with endpoint $V_M$ and 
we call $V_{M'}$ the other endpoint and so on. When we arrive at a minimal element $N$ 
of ${\cal H}$, we add a cherry with stalk $V_N$ and  set of leaves $ N$.

\end{rem}

{\bf Example.} Let $r=6$. Consider  the following hierarchy over $[12]$:
 $${\cal H} =\{\{1,2,3,4,5,6\}, \{4,5,6\},
\{7,8,9,10\},  \{7,8\}, \{9,10\}
\}.$$ The associated $6$-pseudostar is the one in Figure \ref{hierpseudo}, in fact: $L(B)=\{11,12\}$, the maximal elements of ${\cal H}$ are $\{1,2,3,4,5,6\}$ and  $\{7,8,9,10\}$; for $M= \{1,2,3,4,5,6\} $, the set  
$ M- \cup_{H \in {\cal H}, H \subsetneq M} H$ is $\{1,2,3\}$ and the only  element of ${\cal H}$ strictly contained in $M$ is $\{4,5,6\}$, which is minimal in ${\cal H}$; 
for   $M= \{7,8,9,10\}$   the set  $ M- \cup_{H \in {\cal H}, H \subsetneq M} H$ is empty and the only  elements of ${\cal H}$ strictly contained in $M$ are $\{7,8\}$ and $\{9,10\}$, which are minimal in ${\cal H}$.

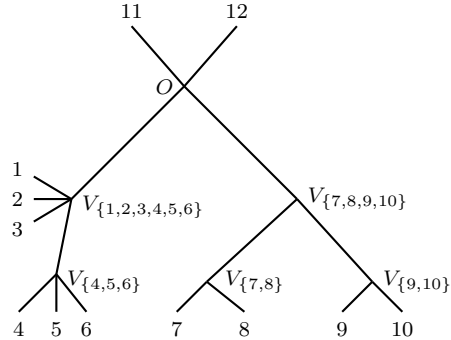
\begin{figure}[h!]
\begin{center}

\begin{tikzpicture}

\draw [thick] (0.5,0) --(-0.2,0.8);
\draw [thick] (0.5,0) --(1.2,0.8); 
\draw [thick] (0.5,0) --(-1,-1.5);
\draw [thick] (0.5,0) --(2,-1.5);
\draw [thick] (-1,-1.5) --(-1.5,-1.5);
\draw [thick] (-1,-1.5) --(-1.5,-1.2);
\draw [thick] (-1,-1.5) --(-1.5,-1.8);
\draw [thick] (-1,-1.5) --(-1.2,-2.5);
\draw [thick] (-1.2,-2.5) --(-1.7,-3);
\draw [thick] (-1.2,-2.5) --(-1.2,-3);
\draw [thick] (-1.2,-2.5) --(-0.8,-3);
\draw [thick] (2,-1.5) --(3,-2.6);
\draw [thick] (2,-1.5) --(0.8,-2.6);
\draw [thick] (3,-2.6) --(2.6,-3);
\draw [thick] (3,-2.6) --(3.4,-3);
\draw [thick] (0.8,-2.6) --(1.3,-3);
\draw [thick] (0.8,-2.6) --(0.4,-3);

\node[left] at (0.5,0) {\scriptsize $O$};
\node[right] at (-1,-1.6) {\scriptsize $V_{\{1,2,3,4,5,6\}}$};
\node[right] at (2, -1.5) {\scriptsize $V_{\{7,8,9,10\}}$};
\node[right] at (-1.2, -2.6) {\scriptsize $V_{\{4,5,6\}}$};
\node[right] at (0.9, -2.6) {\scriptsize $V_{\{7,8\}}$};
\node[right] at (3, -2.6) {\scriptsize $V_{\{9,10\}}$};
\node [left] at (-1.5,-1.5) {\scriptsize 2};
\node [left] at (-1.5,-1.1) {\scriptsize 1};
\node [left] at (-1.5,-1.9) {\scriptsize 3};

\node at (-0.2,1) {\scriptsize 11};
\node at (1.2,1) {\scriptsize 12};
\node [below] at (-1.7,-3) {\scriptsize 4};
\node [below] at (-1.2,-3) {\scriptsize 5};
\node [below] at (-0.8,-3) {\scriptsize 6};
\node [below] at (0.4,-3) {\scriptsize 7};
\node [below] at (1.3,-3) {\scriptsize 8};
\node [below] at (2.6,-3) {\scriptsize 9};
\node [below] at (3.4,-3) {\scriptsize 10};

\end{tikzpicture}

\caption{How to recover the pseudostar from the hierarchy \label{hierpseudo}}
\end{center}
\end{figure}

%\begin{figure}[!h]
%\centering
%\input{hierarchypseudostar.pdf_t} \label{hierpseudo}
%\caption{  }
%\end{figure}

\medskip

We report now two results (Proposition \ref{starbell}  and Theorem \ref{main}) we proved in \cite{B-R2}, because we need them in the proof of our main result  (Theorem \ref{char}).

\begin{prop} \label{starbell}
  Let $k,n \in \N$ with $ 2 \leq k \leq n-2$.
Let ${\cal P}=(P,w)$ be a weighted tree with $L(P)=[n]$. 

 \smallskip
 
 1)  Let $i,j \in [n]$.  

(1.1) If $i,j $ are neighbours,  then 
$ D_{i,X} ({\cal P}) - D_{j,X} ({\cal P}) $ does not depend on  $ 
X \in {[n] - \{i,j\} \choose  k-1}$.

(1.2) If ${\cal P}$ is  an internal-nonzero-weighted essential  pseudostar of kind $(n,k)$, then also the converse is true.

\smallskip

2)  Let $i,j, x,y \in [n]$. 
%(2.1) If $\langle i,j | l,m \rangle$ holds or $P|_{i,j,l,m}$ is a star, then $$ %D_{i,m,R} ({\cal P})  + D_{j,l,R}({\cal P}) = D_{i,l,R} ({\cal P})  + D_{j,m,R}({\cal P}) $$
%for any $R \in { [n]- \{i,j,l,m\}  \choose k-2}$.
%(2.2) 
Let $ k \geq 4 $ and
 ${\cal P}$  be  an internal-nonzero-weighted essential pseudostar of kind $(n,k)$.  We have that $\langle i,j | x,y  \rangle$ holds if and only if at least one of the following conditions holds:

(a) $\{i,j\} $ and $\{x,y\}$ are complete cherries in $P$,

(b) there exist $S, R \in { [n]- \{i,j,x,y\}  \choose k-2}$ such that  
\begin{equation} \label{dis} D_{i,j,S} ({\cal P})  + D_{x,y,S}({\cal P}) \neq  D_{i,x,S} ({\cal P})  + D_{j,y,S}({\cal P}) .\end{equation}
\begin{equation} \label{dis2} D_{i,j,R} ({\cal P})  + D_{x,y,R}({\cal P}) \neq  D_{i,y,R} ({\cal P})  + D_{j,x,R}({\cal P}) .\end{equation}

%In the general case,  if $ D_{i,X} ({\cal P}) - D_{j,X} ({\cal P}) $ does not depend on  $ 
%X \in {[n] - \{i,j\} \choose  k-1}$, $ k \geq 4$ and $x, y$ are two nodes on the path from $i$ to $j $ 
%such that there is no node in the path from $x$ to $y$,  then the weight of the path from $x$ to $y$ is %$0$. In particular, if $k \geq 4$ and ${\cal P}$ has no edges of weight $0$ 
%if $ D_{i,X} ({\cal P}) - D_{j,X} ({\cal P}) $ does not depend on  $ X \in {[n] - \{i,j\} \choose  k-1}$,
%then $i$ and $j$ are neighbours.

\end{prop}

\begin{thm} \label{main} Let $n, k \in \N$ with $ 3 \leq k \leq n-1$.
Let $\{D_I\}_{I \in {[n] \choose k}}$ be a  family of  real numbers.
If it is l-treelike, then there exists exactly  one internal-nonzero-weighted
essential pseudostar ${\cal P} $ of kind $(n,k)$ realizing the family.
If the family $\{D_I\}_{I \in {[n] \choose k}}$ is  
p-l-treelike, then  ${\cal P} $ is positive-weighted.
\end{thm}

%%%%%%%%%%%%%%%%%%%%%%%%%%%%%%%%%%%%%%%%%%%
%%%%%%%%%%%%%%%%%%%%%%%%%%%%%%%%%%%%%%%%%%

\section{Characterization of treelike families}

In \S2 we established a   relation between pseudostars and hierarchies.
In this section, firstly we associate to any   family of real numbers $\{D_I\}_{I \in {[n] \choose k}}$ a hierarchy (see Definition  \ref{fam-hier}) in such a way that, if the family is the family of $k$-weights of a pseudostar, the hierachy associated to the pseudostar and the one associated to the family coincide. 
Then, in the main theorem (Theorem \ref{char}), 
 we give necessary and 
sufficient conditions for  a family   $\{D_I\}_{I \in {[n] \choose k}}$ to be l-treelike and these conditions involve  the  hierarchy associated to the family.

\begin{rem} \label{buoniono}
Let $T$ be a tree with $L(T)=[n]$. Let $C \subset [n]$.

\begin{enumerate}\itemsep0.1cm 
\item[a] By definition, we have that $C$ is a cherry if and only if, for any $i, j \in C$, $i$ and $j$ are neighbours, and this is true if and only if, for any $i, j \in C$,    
there do not exist $x, y \in [n]- \{i,j\}$ such that $\langle i,x | j , y \rangle $ holds; 
 
 \item[b]  Let $C$ be a complete cherry. Let $i,j \in C$. Then $C$ is good if and  only 
if, for any $ x,y \in [n]-C$, we have that $\langle i,j| x,y \rangle$ holds.
\end{enumerate}

Let $r \in \N $ and  let us define $T^0 = T$ and 
$T^s$ to be the tree obtained from $T^{s-1}$ by pruning the  good cherries of 
cardinality less or equal than $r$. If $J$ is a good cherry of $T^s$, we denote the stalk of $J$, which is a leaf of $ T^{s+1}$, by $\lceil  min (J) \rceil $.
Let $C \subset L( T^{s})$.

\begin{enumerate}\itemsep0.1cm 
\item[c] By definition, we have that $C$   is a cherry of $ T^{s}$  
if and only if, for any $\lceil i \rceil , \lceil j \rceil  \in C$, $\lceil i \rceil$ and $\lceil j \rceil$ are neighbours, and this is true if and only if, for any $\lceil i \rceil , \lceil  j   \rceil \in C$,    
there do not exist $\lceil x \rceil, \lceil  y \rceil\in L (T^{s}) - \{\lceil i \rceil , \lceil j \rceil \}$ such that $\langle \lceil i\rceil, \lceil x \rceil | \lceil  j \rceil ,\lceil  y \rceil\rangle $ holds. This is true if and only if, for any $\lceil i \rceil , \lceil  j   \rceil \in C$,    
there do not exist $\lceil x \rceil, \lceil  y \rceil\in L (T^{s}) - \{\lceil i \rceil , \lceil j \rceil \}$ such that $\langle i,  x |   j  ,  y \rangle $ holds. 

\item[d]  Let $C$ be a complete cherry of $T^s$. Let $ \lceil i \rceil, \lceil j  \rceil \in C$. Then $C$ is good if and  only 
if, for any $ \lceil x \rceil , \lceil y \rceil \in  L (T^{s}) -C$, we have that $\langle\lceil i \rceil , \lceil j
\rceil | \lceil x \rceil, \lceil y \rceil \rangle$ holds. This is true if and only if 
for any $ \lceil x \rceil , \lceil y \rceil \in  L (T^{s}) -C$, we have that $\langle  i  ,  j
 |  x , y  \rangle$ holds.
\end{enumerate}
Resuming, 

a) Let $C \subset [n]$;
 
$ C \mbox{ is a cherry } \Longleftrightarrow \forall \, i,j \in C, \not\exists \,
x, y \in [n]- \{i,j\} \mbox{ such that } \langle i,x | j , y \rangle  \mbox{ holds.}
$ 

b) Let $C$ be a complete cherry;

$ C \mbox{ is good }  \Longleftrightarrow \forall \, i,j \in C,
\forall \, x,y \in [n]-C, \mbox{  we have that } \langle i,j| x,y \rangle
\mbox{ holds}. $

c) Let $C \subset L( T^{s})$;

$C \mbox{   is a cherry of } T^{s}   \Longleftrightarrow 
 \forall \, \lceil i \rceil , \lceil  j   \rceil \in C,    
\not \exists \lceil x \rceil, \lceil  y \rceil\in L (T^{s}) - \{\lceil i \rceil , \lceil j \rceil \} \mbox{ such that } \langle i,  x |   j  ,  y \rangle \mbox{ holds.}$

d) Let $C$ be a complete cherry of $T^s$;

$C \mbox{ is good }  \Longleftrightarrow  \forall \, \lceil i \rceil , \lceil  j   \rceil \in C,   
\forall \,  \lceil x \rceil , \lceil y \rceil \in  L (T^{s}) -C,  \mbox{ we have that } \langle  i  ,  j |  x , y  \rangle \mbox{ holds.}$

\end{rem}

%\begin{rem} Let $n, k \in \N$ with $ 3 \leq k \leq n-2$.
%Let  ${\cal P}=(P,w)$ be an internal-nonzero-weighted essential  
%$(n--k)$-pseudostar  with  $L(P)=[n]$
%and let us denote $D_I ({\cal P}) $ by $ D_I $ for any $ I \in {[n] \choose %k}$. 
%Observe that  the hierarchy $ {\cal H}$  over $[n]$ defined  by $P$ as in %Definition \ref{hier} can be recovered from  the family $\{D_I\}_I$  as %follows:

%$\bullet $ by part 1 of Proposition \ref{starbell},
%the  complete cherries of $P$ are exactly the maximal  subsets 
%$C$ of $[n]$ such 
%that $ D_{i,X} ({\cal P}) - D_{l,X} ({\cal P}) $ does not depend on  $ 
%X \in {[n] - \{i,l\} \choose  k-1}$ for any $i, l \in C$;

%$\bullet $ we can discover which of them are 
 %good  (thus, which are the minimal elements of  ${\cal H}$) by using %Remark \ref{buoniono} (b)  and part 2 of  Proposition \ref{starbell}; 

%$\bullet$ we can discover which subsets of $L(P^{s}) $ are complete %cherries by using Remark \ref{buoniono} (c) and part 2 of  Proposition %\ref{starbell}; 

%$\bullet$ we can discover which  of them are good by using  Remark %\ref{buoniono} (d) and part 2 of  Proposition \ref{starbell}; 

%\end{rem}

\begin{defin}    \label{fam-hier}   Let $n, k \in \N$ with $ 5 \leq k \leq n-1$. 
Let $\{D_I\}_{I \in {[n] \choose k}}$ be a  family in $\R$.
 
Let $${\cal C}^0=\left\{  Z \in { [n] \choose
           \geq 2 } \;|\;      \vspace{0.2cm} \\  D_{i,X}  - D_{j,X}  \mbox{ does not depend on } 
X \in {[n] - \{i,j\} \choose  k-1} \; \forall i, j \in Z 
\right\},$$ let $\underline{\cal C}^0$ be the set of the maximal elements of
${\cal C}^0$  
and $${\cal G}^0= \left \{ \begin{array}{ll} Z \in \underline{\cal C}^0 \; | \; & \# Z \leq n-k \mbox{ and }
 \forall \, i,j \in Z, \;\forall \,   x,y \in [n]-Z 
\mbox{ one of the following holds: } \vspace*{0.2cm}  \\ 
&
(a) \;   \{i,j\}, \{x,y\} \in \underline{\cal C}^0 \\
 & (b)  \;\exists \, R,S \in { [n]- \{i,j,x,y\}  \choose k-2} \mbox{ s.t. }  \left\{
 \begin{array}{c}D_{i,j,R}   + D_{x,y,R} \neq  D_{i,x,R}   + D_{j,y,R} \\
 D_{i,j,S}   + D_{x,y,S} \neq  D_{i,y,S}   + D_{j,x,S} \end{array} \right.
 \end{array}\right\}.$$
   Let $[n]^0= [n]$.
     We define inductively $[n]^s$, ${\cal C}^s$, $\underline{\cal C}^s$, ${\cal G}^s$ 
     as follows:
     for   $s \geq 1$, we define $[n]^s$ to be the set obtained from $[n]^{s-1}$ by eliminating for every $Z \in {\cal G}^{s-1}$ all the elements
     of $Z$  apart from the minimum 
 $${\cal C}^s= \left\{ \begin{array}{ll}
 Z  \in {[n]^s \choose \geq 2}  \; | \;&  
      \forall \, i,j \in Z, \; \forall \, x,y \in [n]^s-\{i,j\}   \mbox{  both the following do not hold: }   
 \\  & \vspace{0.2cm} \begin{array}{l} 
 (a) \;   \{i,x\}, \{j,y\} \in \underline{\cal C}^0\\ 
 (b)  \;\exists\, R,S \in { [n]- \{i,j,x,y\}  \choose k-2} \mbox{ s.t. }  
 \left\{\begin{array}{c}D_{i,x,R}   + D_{j,y,R} \neq  D_{i,j,R}   + D_{x,y,R} \\
 D_{i,x,S}   + D_{j,y,S} \neq  D_{i,y,S}   + D_{j,x,S} \end{array} \right.
 \end{array}         
\end{array}\right\},$$ let $\underline{\cal C}^s$ be the set of the maximal elements of
${\cal C}^s$ 
and  $${\cal G}^s= \left \{ \begin{array}{ll} Z \in \underline{\cal C}^s \; | \; & 
\# \partial Z \leq n-k \mbox{ and }  \forall \, i,j \in Z, \; \forall \,  x,y \in [n]-Z 
\mbox{ one of the following holds: } \vspace*{0.2cm}  \\ 
&
(a) \;   \{i,j\}, \{x,y\} \in \underline{\cal C}^0 \\
 & (b)  \;\exists R,S \in { [n]- \{i,j,x,y\}  \choose k-2} \mbox{ s.t. }  \left\{
 \begin{array}{c}D_{i,j,R}   + D_{x,y,R} \neq  D_{i,x,R}   + D_{j,y,R} \\
 D_{i,j,S}   + D_{x,y,S} \neq  D_{i,y,S}   + D_{j,x,S} \end{array} \right.
 \end{array}\right\},$$

 where: 
 we  say  $y_0 \in [n]$ descends from $y_s \in [n]^s$ if and only if 
  there exist  (not necessarily distinct)  $y_1, ...., y_{s-1} \in [n]$ and, for any   $t=0,...., s-1$,   an element of ${\cal G}^s$ containing both $y_t$ and $y_{t+1}$;
 for any $Z \in G^s$, let $\partial Z$ be  the set of the $y\in [n]$
 descending from some element of $Z$ and let 
  $\partial {\cal G}^s= \{ \partial Z\,|\, Z \in  {\cal G}^s\}$
 
Finally, we define $${\cal H} = \cup_{s \geq 0} \partial {\cal G}^s$$
and we call ${\cal H}$  {\bf the  hierarchy associated to the family}
                        $\{D_I\}_{I \in {[n] \choose k}}$. \end{defin}

Let $n, k \in \N$ with $ 3 \leq k \leq n-2$.
Let  ${\cal P}=(P,w)$ be an internal-nonzero-weighted essential  
pseudostar  of kind $(n,k)$ with  $L(P)=[n]$
and let us denote $D_I ({\cal P}) $ by $ D_I $ for any $ I \in {[n] \choose k}$. 
Observe that, by Remark \ref{buoniono} and Proposition \ref{starbell},  the hierarchy $ {\cal H}$  over $[n]$ defined  by $P$ as in Definition \ref{hier} is equal to the hierarchy associated to the family  $\{D_I\}_I$; precisely ${\cal C}^s$ is the set of the cherries of the tree $P^s$ in Definition \ref{hier}, 
 $\underline{\cal C}^s$ is the set of the complete cherries of $P^s$, and
 ${\cal G}^s$ is the set of the good cherries of $P^s$. 

\begin{rem} \label{importante}  Let $n, k \in \N$ with $ 2 \leq k \leq n-2$.
Let ${\cal P} =(P,w) $ be  a weighted pseudostar of kind $(n,k)$
with $L(P)=[n]$. Let ${\cal H}$
be a hierarchy  over $[n]$ associated to $P$as in Definition \ref{hier}. 
Observe that, for any $J \in {\cal H}$ and any $I \in { [n] \choose k}$, the subtree realizing  $D_I ({\cal P})  $ contains $e_J$ if and only if 
$ I \cap J \neq \emptyset$ and $I \not \subset J$.
\end{rem}

We are ready now to state the characterization of treelike families.
In the proof,  it will be necessary to use two technical lemmas; we postpone them to the appendix.

\begin{thm} \label{char}
Let $n, k \in \N$ with $ 5 \leq k \leq n-1$. 
Let $\{D_I\}_{I \in {[n] \choose k}}$ be a  family of  real numbers.

If $k \leq n-2$, the family $\{D_I\}_{I }$
 is l-treelike if and only if the hierarchy ${\cal H }$ over $[n]$ 
associated to the family $\{D_I\}_{I \in {[n] \choose k}}$ is such that:
   
%(i) the  clusters of ${\cal H}$ have cardinality between $2$ and $n-k$,    

(i)  if ${\cal H}$ covers [n], then the number of the maximal clusters
 of ${\cal H}$ is not $2$,

(ii) for any $q \in  \{1,...., n-1\}$, $s \in \{1, k-1\}$ and for any  $W,W' \in {[n]  \choose  s}$
$$ \sum_{i=1,...,q} D_{W,Z_i} - D_{W', Z_i} $$ 
does not depend on $Z_i \in {[n] -W-W' \choose  k-s}$ under the condition that, in the free $\Z$-module $\oplus_{H \in {\cal H}} \Z H$, the sum
$$ \sum_{i=1,...,q} \left[ \sum_{H \in {\cal H}, H \cap (W Z_i) \neq \emptyset, \; H \not \supset (W Z_i ) } H -
 \sum_{H \in {\cal H}, \; H \cap (W' Z_i )\neq \emptyset, \; H \not \supset (W' Z_i )\ } H \right] $$ 
does not change.

If $k = n-1$, the family $\{D_I\}_{I }$ is always  l-treelike.
\end{thm}

\begin{proof}
If $k=n-1$, it is easy to show that there exists a weighted star realizing the family $\{D_I\}_I$.

Suppose $k \leq n-2$.
If the family $\{D_I\}_I $ is l-treelike,  then there exists a weighted pseudostar of kind $(n,k)$
realizing it by Theorem \ref{main}; it induces a hierarchy  over $[n]$  as in Definition \ref{hier}  and it is easy to see that conditions (i) and (ii) hold; by Remark \ref{importante}, we have also that condition (iii) holds. 

Suppose now that the hierarchy ${\cal H}$ over $[n]$ associated to
the family $\{D_I\}_{I \in {[n] \choose k}}$ 
satisfies (i) and (ii).  Let $P$ be the essential pseudostar of kind $(n,k)$ determined by 
${\cal H}$ (see Remark \ref{fromhiertops}); observe that it is essential by condition (i).  For any $J \in {\cal H}$, 
let $e_J$ be defined as in  Remark \ref{fromhiertops}; we define
\begin{equation} \label{w_V}
w(e_J) := D_{a,X} - D_{a' ,X}  - D_{a ,X'}   + D_{a'X'},
\end{equation}
for any $a, a' \in [n], X,X' \subset [n]$  such that  $a, a' \not \in X, X' $ 
and 
\begin{equation} \label{formal}
 \sum_{\stackrel{\mbox{\scriptsize $H \in {\cal H}$,}}{ \;H \cap (aX) \neq \emptyset, \; H \not \supset (aX)}}
H \;- 
\sum_{\stackrel{\mbox{\scriptsize $H \in {\cal H}$,}}{ \;H \cap (a'X) \neq \emptyset, \; H \not \supset (a'X)}}
H  \;
- 
\sum_{\stackrel{\mbox{\scriptsize $H \in {\cal H}$,}}{ \;H \cap (aX') \neq \emptyset, \; H \not \supset (aX')}}
H  \;
+
\sum_{\stackrel{\mbox{\scriptsize $H \in {\cal H}$,}}{ \;H \cap (a'X') \neq \emptyset, \; H \not \supset (a'X')}}
H  \end{equation}
is equal to $J$. 
Let us check that the definition of $w(e_J)$  is a good definition: 
\begin{itemize}\itemsep0.5pt
\item
to see that it does not depend on $X$, it is sufficient 
to see that $D_{a,X} - D_{a' ,X}$ does not depend on $X$ under the condition that the  sum in (\ref{formal}) does not depend on $X$; obviously this is equivalent to the fact that 
$ \sum_{ H \in {\cal H},  H \cap (aX) \neq \emptyset, H \not\supset (aX) }   H 
-  \sum_{ H \in {\cal H},  H \cap (a'X) \neq \emptyset,  H \not\supset (a'X)}   H $ does not depend on $X$; so our assertion
follows from condition (ii) by taking $q=1$, $s=1$, $W=\{a\}$, $W'=\{a'\}$ and $Z_1=X$;
in an analogous way we can see that it does not depend on $X'$;
\item to see that it does not depend on $a$, it is sufficient 
to see that $D_{a,X} - D_{a,X'}$ does not depend on $a$ under the condition that the sum in (\ref{formal}) does not depend on $a$; obviously this is equivalent to the fact that 
$ \sum_{ H \in {\cal H},  H \cap (aX) \neq \emptyset,  H \not\supset (aX)}   H 
-  \sum_{ H \in {\cal H},  H \cap (aX') \neq \emptyset, H \not\supset (aX')}   H $ does not depend on $a$; so our assertion
follows from condition (ii) by taking $q=1$, $s=k-1$, $W=X$, $W'=X'$,  and $Z_1=\{a\}$;
in an analogous way we can see that it does not depend on $a'$.
\end{itemize}
Moreover, observe that,  by Lemma \ref{aa'XX'},  it is possible to find $a,a',X,X'$ as required.

For any $i \in [n] $, we define the weight of the twig $e_i$ as follows:  
\begin{equation} \label{w_i}
w(e_i)   :=\frac{1}{k}
\left[ D_I  + \sum_{l \in I } \left(D_{i,X(i,l)}  - D_{l,X(i,l)}  \right) - \sum_{J \in  {\cal H}, \; J   \cap I \neq \emptyset, \; J \not \supset I } 
w(e_J) \right]
\end{equation}
for any $ I  \in {[n] \choose k}$,  $X(i,l)  \in {[n]-\{i,l\} \choose k-1}$  such that
$$ \sum_{H \in {\cal H}, \;H \cap (i,X(i,l)) \neq \emptyset, \; H \not \supset (i,X(i,l))}
H\;\; \;- \sum_{H \in {\cal H}, \;H \cap (l,X(i,l)) \neq \emptyset, \; H \not \supset (l ,X(i,l))}
H  \,\;= \;\;0.$$ 
Observe that, by Lemma \ref{X(i,l)}, it is possible to find  $X(i,l)$ as required.
The definition of $w(e_i)$ does not depend on the choice of $X(i,l)
$ by condition (ii); we have to show that it does not depend
 on $I$. Let $I=(a,Y) $ and $I' =(a',Y)$  for some distinct $a,a' \in [n]$,  $ Y \in { [n]-\{a,a'\} \choose k-1}$. We have to show that 
$$D_{a, Y} + \sum_{l \in (aY) } \left(D_{i,X(i,l)}  - D_{l,X(i,l)}  \right) - \sum_{J \in  {\cal H},  J   \cap (aY) \neq \emptyset,  J \not\supset (aY)  }  w(e_J) =$$ 
$$=D_{a' ,Y} + \sum_{l \in (a'Y)} \left(D_{i,X(i,l)}  - D_{l,X(i,l)}  \right) - \sum_{J \in  {\cal H},  J   \cap (a'Y) \neq \emptyset,  J \not\supset (a'Y)  }  w(e_J) ,$$ 
that is 
$$D_{a, Y} + D_{i,X(i,a)}  - D_{a,X(i,a)}   - \!\!\!\! \sum_{\stackrel{\scriptsize \mbox{$J \in  {\cal H}$,}}{  J   \cap (aY) \neq \emptyset, J \not\supset (aY) } }  \!\!\!\!  w(e_J) =
D_{a', Y} +  D_{i,X(i,a')}  - D_{a',X(i,a')}   -   \!\!\!\! \sum_{\stackrel{\scriptsize \mbox{$J \in  {\cal H}$,}}{   J   \cap (a'Y) \neq \emptyset, J \not\supset (a'Y)}  }   \!\!\!\! w(e_J) .$$

Observe that  $\{J \in  {\cal H}| \,  J  \cap (aY) \neq \emptyset,\, J \not\supset (aY) \} $ can be written as disjoint union of the following sets:
$$\begin{array}{c}
\{J \in  {\cal H}| \, J \ni a, \, J \not \ni a', \,  J   \cap (aY) \neq \emptyset,\, J \not\supset (aY) \} ,
\{J \in  {\cal H}| \, J  \not\ni a, \, J \ni a', \,  J   \cap (aY) \neq \emptyset,\, J \not\supset (aY) \}  ,
\\
\{J \in  {\cal H}| \, J  \ni a, \, J \ni a', \,  J   \cap (aY) \neq \emptyset,\, J \not\supset (aY) \} ,
\{J \in  {\cal H}| \, J \not \ni a, \, J  \not \ni a', \,  J   \cap (aY) \neq \emptyset,\, J \not\supset (aY) \}, 
\end{array}$$ that is, as disjoint union of 
$$\begin{array}{c}
\{J \in  {\cal H}| \, J \ni a, \, J \not \ni a', \, J \not\supset Y \} ,\; \;
\{J \in  {\cal H}| \, J  \not\ni a, \, J \ni a', \,  J   \cap Y \neq \emptyset\}  ,
\\
\{J \in  {\cal H}| \, J  \ni a, \, J \ni a', \,   J \not\supset Y \} ,\;\;
\{J \in  {\cal H}| \, J \not \ni a, \, J  \not \ni a', \,  J   \cap Y \neq \emptyset\} ,
\end{array}$$ and then as disjoint union of 
$$\begin{array}{ll}
\{J \in  {\cal H}| \, J \ni a, \, J \not \ni a', \, J   \cap Y \neq \emptyset, \, J \not\supset Y \} ,&
\{J \in  {\cal H}| \, J  \not\ni a, \, J \ni a', \,  J   \cap Y \neq \emptyset, \, J \not\supset Y\}  ,
\\\{J \in  {\cal H}| \, J \ni a, \, J \not \ni a', \, J   \cap Y = \emptyset\} ,&
\{J \in  {\cal H}| \, J  \not\ni a, \, J \ni a', \,  J   \supset Y\}  ,
\\
\{J \in  {\cal H}| \, J  \ni a, \, J \ni a', \,   J \not\supset Y \} ,&
\{J \in  {\cal H}| \, J \not \ni a, \, J  \not \ni a', \,  J   \cap Y \neq \emptyset\} ,
\end{array}$$ 
Analogously we can write  $\{J \in  {\cal H}| \,  J  \cap (a'Y) \neq \emptyset,\, J \not\supset (a'Y) \} $.

%Let $X$ be a $(k-1)$-subset of $ [n]-\{a,a',i\}$ containing an element
%in the minimal ${\cal H}$-cluster containg $i$,  an element
%in the minimal ${\cal H}$-cluster containg $a$ and an element
%in the minimal ${\cal H}$-cluster containg $a'$;

Let us  take both $ X(i,a)$  and $ X (i,a')$  equal to a set $X$   satisfying the conditions of Lemma \ref{X(i,l)}  for $i,a$, for  $i,a'$ and for 
$a,a'$  (there exists since $k\geq 5$). By simplifying,
 the assertion becomes 
$$D_{a ,Y}   - D_{a,X}   - \!\! \! \sum_{ 
\stackrel{ \mbox{\scriptsize $J \in  {\cal H}$  and } }{
\stackrel{ \mbox{\scriptsize either   $J \ni a, J \not \ni a', J \cap Y = \emptyset$} }{ 
\mbox{\scriptsize or $\, J \ni a', J \not \ni a, \,Y \subset J$ }  }}
} %sum_{ 
 w(e_J) \;\;\;= \;\;\;
D_{a' ,Y}  - D_{a',X}   - \!\!\! \sum_{ 
\stackrel{ \mbox{\scriptsize $J \in  {\cal H}$  and } }{
\stackrel{ \mbox{\scriptsize either   $J \ni a', J \not \ni a, J \cap Y = \emptyset$} }{ 
\mbox{\scriptsize or $\, J \ni a, J \not \ni a',\, Y \subset J$ }  }}
} %sum_{ 
 w(e_J) . $$ 
For any $J \in  {\cal H}$ such that $  J \ni a', J \not \ni a $, and   $J \cap Y = \emptyset $ or 
$Y \subset J$,
 let $Z_J, Z'_J$ be such that  the sum 
$$ \sum_{\stackrel{\mbox{\scriptsize $H \in {\cal H}, H \cap (a' Z_J) \neq \emptyset$}}{\mbox{\scriptsize $H \not \supset (a' Z_J)$} }}
 H  -
\sum_{\stackrel{\mbox{\scriptsize $H \in {\cal H}, H \cap (a Z_J) \neq \emptyset$}}{\mbox{\scriptsize $H \not \supset (a Z_J)$} }}
 H  -
\sum_{\stackrel{\mbox{\scriptsize $H \in {\cal H}, H \cap (a' Z'_J) \neq \emptyset$}}{\mbox{\scriptsize $H \not \supset (a' Z'_J)$} }} H  +
\sum_{\stackrel{\mbox{\scriptsize $H \in {\cal H}, H \cap (a Z'_J) \neq \emptyset$}}{\mbox{\scriptsize $H \not \supset (a Z'_J)$} }} H  $$ is equal to $J$.
By  the definition in (\ref{w_V}),
we have that $$w(e_J)= D_{a', Z_J} - D_{a ,Z_J} - D_{a' ,Z'_J} + D_{a, Z'_J}.$$ 
For any $J \in  {\cal H}$ such that $  J \ni a, J \not \ni a'$, and $ J \cap Y = \emptyset  $
or $ Y \subset J$,   let $R_J, R'_J$ 
be such that  the sum 
$$ \sum_{\stackrel{\mbox{\scriptsize $H \in {\cal H}, H \cap (a R_J) \neq \emptyset$}}{\mbox{\scriptsize $H \not \supset (a R_J)$} }} H  -
\sum_{\stackrel{\mbox{\scriptsize $H \in {\cal H}, H \cap (a' R_J) \neq \emptyset$}}{\mbox{\scriptsize $H \not \supset (a' R_J)$} }}
H  -
\sum_{\stackrel{\mbox{\scriptsize $H \in {\cal H}, H \cap (a R'_J) \neq \emptyset$}}{\mbox{\scriptsize $H \not \supset (a R'_J)$} }} H  +
\sum_{\stackrel{\mbox{\scriptsize $H \in {\cal H}, H \cap (a' R'_J) \neq \emptyset$}}{\mbox{\scriptsize $H \not \supset (a' R'_J)$} }} H  $$ is equal to $J$; 
by  the definition in (\ref{w_V}),
we have that  $$w(e_J)= D_{a, R_J} - D_{a', R_J} - D_{a, R'_J} + D_{a' ,R'_J}.$$ 
So our assertion becomes
\begin{equation} \label{ass}
\begin{array}{ll}
D_{a ,Y}   - D_{a', Y}  \;\; =  & \;\;\; D_{a,X} - D_{a',X}   
\vspace*{0.05cm}\\  
& - \sum_{J \in  {\cal H},  J \cap Y = \emptyset ,
 J \ni a', J \not \ni a  }  (D_{a', Z_J} - D_{a ,Z_J} - D_{a', Z'_J} + D_{a ,Z'_J}) \\
\vspace*{0.2cm} &
- \sum_{J \in  {\cal H},  Y \subset J , J \ni a, J \not \ni a'  }  
(D_{a, R_J} - D_{a', R_J} - D_{a, R'_J} + D_{a' ,R'_J} )
 \\
& + \sum_{J \in  {\cal H},   Y \subset J,
 J \ni a', J \not \ni a  }  (D_{a', Z_J} - D_{a ,Z_J} - D_{a', Z'_J} + D_{a ,Z'_J}) \\
\vspace*{0.2cm}
& + \sum_{J \in  {\cal H},  J \cap Y = \emptyset , J \ni a, J \not \ni a'  }  
(D_{a, R_J} - D_{a', R_J} - D_{a, R'_J} + D_{a' ,R'_J} ),  
\end{array}
\end{equation}
that is 
\begin{equation} \label{ass2}
\left(\begin{array}{r}
D_{a ,Y}   - D_{a', Y}     
\\  + \sum_{J \in  {\cal H},  J \cap Y = \emptyset , J \ni a', J \not \ni a  }  (  D_{a, Z'_J} - D_{a' ,Z'_J})   
\\ +\sum_{J \in  {\cal H},  Y \subset J , J \ni a, J \not \ni a'  }  (  D_{a, R_J} - D_{a' ,R_J} )
\\ + \sum_{J \in  {\cal H},   Y \subset J, J \ni a', J \not \ni a  }  (D_{a, Z_J} - D_{a' ,Z_J} ) 
\\ + \sum_{J \in  {\cal H},  J \cap Y = \emptyset , J \ni a, J \not \ni a'  }  ( D_{a, R'_J} - D_{a' ,R'_J} )  
\end{array}\right) =
\left(\begin{array}{l}
  D_{a,X} - D_{a',X}   
\\ +  \sum_{J \in  {\cal H},  J \cap Y = \emptyset , J \ni a', J \not \ni a  }  (D_{a, Z_J} - D_{a' ,Z_J} ) 
\\ + \sum_{J \in  {\cal H},  Y \subset J , J \ni a, J \not \ni a'  }  (D_{a, R'_J} - D_{a', R'_J}  )
\\ + \sum_{J \in  {\cal H},   Y \subset J, J \ni a', J \not \ni a  }  ( D_{a, Z'_J} - D_{a' ,Z'_J}) 
\\ + \sum_{J \in  {\cal H},  J \cap Y = \emptyset , J \ni a, J \not \ni a'  }  (D_{a, R_J} - D_{a', R_J}  ) . 
\end{array}\right)
\end{equation}

 Observe that 
$$\# ( \{J \in  {\cal H}| \;
 J \ni a', J \not \ni a \} \cup \{J \in  {\cal H}|\;  
 J \ni a, J \not \ni a' \}) \leq n-2, $$
in fact: let $$x := \#  \{J \in  {\cal H}|\;  J \ni a', J \not \ni a \},\;\;\;\;\;\;\;y := \#  \{J \in  {\cal H} |\;J \ni a, J \not \ni a' \};$$ 
the set $\{J \in  {\cal H}| \; J \ni a', J \not \ni a \}$ is  a chain, so in its  largest ${\cal H}$-cluster,  call it $A$, there are at least $x+1$ elements; analogously in the largest ${\cal H}$-cluster  contained in  $\{J \in  {\cal H}| \; J \ni a', J \not \ni a \}$, call it $B$, there are at least $y+1$ elements; since $A$ and $B$ are  disjoint, we have that $$ (x+1) + (y +1)  \leq n,$$  thus
 $x+y \leq n-2$, as we wanted to prove.  Hence the number of the terms
at each member of (\ref{ass2}) is at most $n-1$.
Therefore it is easy to see that our assertion (\ref{ass2}) follows from 
 condition~(ii): write it as (\ref{ass}) and observe that
the  sum $$ \sum_{H \in {\cal H}, \;  H \cap (aX)  \neq \emptyset, \;  H \not \supset (aX) } H
- \sum_{H \in {\cal H} , \; H \cap (a'X)  \neq \emptyset, \;  H \not \supset (a'X) } H $$ is $0$ for the definition of $X$.

So we have defined the weight of $e_i$ for every $i \in [n]$ and the weight of $e_J$ for every $J \in {\cal H}$.

Let ${\cal P}=(P,w)$, where $w$ is the weight we have just defined.
We have to show that $D_I ({\cal P}) =D_I $ for any $I \in {[n] \choose k}$.
First we show that, for any $i,j \in [n]$,  
\begin{equation} \label{weiwej} w(e_i) -w (e_j) =D_{i, X(j,i)} - D_{j, X(j,i)},
\end{equation}
for any $X(i,j) $ such that
$$\sum_{\stackrel{ \mbox{\scriptsize $H \in {\cal H}, H \cap (j,X(j,i))  \neq \emptyset$}}{ 
\mbox{\scriptsize $H \not\supset (j,X(j,i))$} } } 
 H
-\sum_{\stackrel{ \mbox{\scriptsize $H \in {\cal H}, H \cap (i,X(j,i))  \neq \emptyset$}}{ 
\mbox{\scriptsize $H \not\supset (i,X(j,i))$} } } 
 H=0. $$
Let us choose the same $I$ in the definition of $ w(e_i)$ and $  w (e_j) $ (see  (\ref{w_i})) and let us choose it containing neither $i$ nor $j$; so  we get 
$$ w(e_i) -w (e_j) = \frac{1}{k} \left[ 
\sum_{t \in I} \left( D_{i ,X(t,i)} - D_{t, X(t,i)} \right)-  
\sum_{t \in I } \left( D_{j, X(t,j)} - D_{t, X(t,j)} \right) 
\right] = $$ 
$$ =\frac{1}{k} \left[ 
\sum_{t \in I } 
\left( D_{i, X(t,i)} - D_{t ,X(t,i)} - D_{j, X(t,j)} + D_{t, X(t,j)} \right)  
\right].$$
For any $t \in I$,  take $X(t,i)$ and $ X(t,j)$ equal to a set $X_t$  satisfying the conditions of Lemma \ref{X(i,l)}  for the couple  $t,i$, for the couple $t,j$ and for the couple $i,j$ (there exists since $k \geq 5$). So we get 
$$ w(e_i) -w (e_j) =
\frac{1}{k} \left[ 
\sum_{t \in I } \left( D_{i, X_t} - D_{j ,X_t} \right) 
\right] .$$ %= D_{i, X(j,i)} - D_{j, X(j,i)}.$$
Moreover, by condition (ii),   we have that 
$ D_{j ,X_t}  - D_{i, X_t} = D_{j, X(j,i)} - D_{i, X(j,i)}$ for any $t \in I$, since 
$$ \sum_{\stackrel{ \mbox{\scriptsize $H \in {\cal H}, H \cap (j,X_t))  \neq \emptyset$}}{ 
\mbox{\scriptsize $H \not\supset (j,X_t)$} } }  H
-\sum_{\stackrel{ \mbox{\scriptsize $H \in {\cal H}, H \cap (i,X_t)  \neq \emptyset$}}{ 
\mbox{\scriptsize $H \not\supset (i,X_t)$} } } 
 H \;\;= \,\; 0.$$
 Hence we get (\ref{weiwej}).
 
Obviously, for any $ I \in  {[n] \choose k}$, we have that
$$ D_I ({\cal P}) = \sum_{l \in I} w(e_l) + \sum_{J \in  {\cal H}, \, J   \cap I \neq \emptyset, \, 
J \not \supset I} 
w(e_J). $$ 
So, for any $i \in I$,
$$ w(e_i) = \frac{1}{k} \left[  D_I ({\cal P}) + \sum_{l \in I } \left( w(e_i) 
 - w(e_l) \right) - \sum_{J \in  {\cal H},\,  J   \cap I \neq \emptyset,\, 
J \not \supset I } 
w(e_J) 
\right] ,$$ 
which, by (\ref{weiwej}), is equal to 
$$ w(e_i) = \frac{1}{k} \left[  D_I ({\cal P}) + \sum_{l \in I } \left(
D_{i ,X(l,i)} - D_{l, X(l,i)}   \right) - \sum_{J \in  {\cal H}, \, J   \cap I \neq \emptyset,\, 
J \not \supset I } w(e_J) \right] .$$ 
On the other side we have defined $w(e_i) $ to be 
$$ \frac{1}{k} \left[  D_I  + \sum_{l \in I } \left(
D_{i, X(l,i)} - D_{l, X(l,i)}   \right) - \sum_{J \in  {\cal H}, \, J   \cap I \neq \emptyset, \, 
J \not \supset I} w(e_J) \right] ,$$ 
so we get $D_I ({\cal P}) =D_I $ for any $I$.
\end{proof}

\begin{rem} \label{charpos} Let $n, k \in \N$ with $ 2 \leq k \leq n-2$.
Let $\{D_I\}_{I \in {[n] \choose k}}$ be a  family of positive  real numbers.
Obviously the family  $\{D_I\}_I$ is p-l-treelike if and only if 
there  exists a hierarchy ${\cal H}$ over $[n]$ such that 
the conditions (i) and (ii) of Theorem \ref{char} hold, and, in addition,
the numbers in (\ref{w_V}) and (\ref{w_i}) are positive for any $i \in [n], J \in {\cal H}$.
\end{rem}

\section{Appendix}

\begin{lem} \label{X(i,l)}
Let $k, n \in \N$ with $ 4 \leq k \leq n-2$. Let ${\cal H}$ be a hierarchy on $[n]$ such that  its clusters have cardinality less than or equal to $n-k$ and greater than or equal to $2$. For any $ t \in \cup_{H \in {\cal H}} H$, 
denote by $m_t$  the minimal $\cal{H}$-cluster containing $t$
and by $M_t$  the maximal $\cal{H}$-cluster containing $t$. 
Let $i,l \in [n]$ and
 $ X \in {[n]-\{i,l\} \choose k-1}$ satisfy the following conditions:

\begin{itemize}
\item[$\bullet$] if $i,l \in \cup_{H \in {\cal H}} H$ and $M_i \cap M_l= \emptyset $, then 
	\vspace*{-0.1cm}
	\begin{itemize}\itemsep0.1pt
	\item[] $X$ contains an element $\overline{i} \in m_i$ different from $i$,
 	\item[] $X$ contains an element $\overline{l} \in m_l$ different from $l$;
	\end{itemize}
	\vspace*{-0.2cm}

\item[$\bullet$] if $i,l \in \cup_{H \in {\cal H}} H$ and $M_i \cap M_l\neq \emptyset $ then
	\vspace*{-0.1cm}
	\begin{itemize}
	\item[$\centerdot$] if $m_i \subset m_l$: 
		\vspace*{-0.1cm}
		\begin{itemize}\itemsep0.1pt
		\item[] $X$ contains an element $\overline{i} \in m_i$ different from $i$,
		\item[] $X$ contains an element $\hat{i} \in [n]-M_i;$
		\end{itemize}
	\item[$\centerdot$] if $m_l \subset m_i$:
		\vspace*{-0.1cm}
		\begin{itemize}\itemsep0.1pt
		\item[] $X$ contains an element $\overline{l} \in m_l$ different from $l$,
		\item[] $X$ contains an element $\hat{l} \in [n]-M_l;$
		\end{itemize}
	\item[$\centerdot$] if $m_i \cap m_l=\emptyset$:
		\vspace*{-0.1cm}
		\begin{itemize}\itemsep0.1pt
		\item[] $X$ contains an element $\overline{i} \in m_i$ different from $i$,
		\item[] $X$ contains an element $\overline{l} \in m_l$ different from $l$,
		\item[] $X$ contains an element $\hat{i} \in [n]-M_i;$
		\end{itemize}
	\end{itemize}
	\vspace*{-0.2cm}

\item[$\bullet$] if $i \in \cup_{H \in {\cal H}} H$ and $l \not\in \cup_{H \in {\cal H}} H$, then 
	\vspace*{-0.1cm}
	\begin{itemize}\itemsep0.1pt
	\item[] $X$ contains an element $\overline{i} \in m_i$ different from $i$,
	\item[] $X$ contains an element $\hat{i} \in [n]-M_i$;
	\end{itemize}
	\vspace*{-0.2cm}

\item[$\bullet$] if $l \in \cup_{H \in {\cal H}} H$ and $i \not\in \cup_{H \in {\cal H}} H$, then
	\vspace*{-0.1cm}
	\begin{itemize}\itemsep0.1pt
	\item[] $X$ contains an element $\overline{l} \in m_l$ different from $l$,
	\item[] $X$ contains an element $\hat{l} \in [n] - M_l.$
	\end{itemize}
\end{itemize}
\vspace*{-0.2cm}

Then, in the free $\Z$-module $\oplus_{H \in {\cal H}} \Z H$,  $$ \sum_{H \in {\cal H}, \;H \cap (iX) \neq \emptyset, \; H \not \supset (iX)}
H\;\; \;- \sum_{H \in {\cal H}, \;H \cap (lX) \neq \emptyset, \; H \not \supset (lX)}
H  \,\;= \;\;0$$ 
\end{lem}

\begin{proof} 
We have to show that, for every $V \in {\cal H}$, we have that 
$V \cap (iX) \neq \emptyset$ and $ V \not \supset (iX)$ if and only if 
$V \cap (lX) \neq \emptyset$ and $ V \not \supset (lX)$.
We have five possible cases.

$\bullet$ $V \cap X= \emptyset $.

We want to prove that, 
in this case, we have that $V \cap (iX) = \emptyset $. Suppose on the contrary  that  $V \cap (iX) \neq \emptyset $; hence $i \in  V $ and then, obviously, $i \in \cup_{H \in {\cal H}} H$.
If $l \in \cup_{H \in {\cal H}} H$, $M_i \cap M_l \neq \emptyset$ 
and $m_l\subset m_i,$ then $\overline{l} \in X$ by assumption; 
by definition, we have that $\overline{l} \in m_l $ and, 
since   $m_l \subset m_i \subset V$, we have $\overline{l} \in V$ and thus $X \cap V \neq \emptyset$, which is absurd. 
In the other cases, by assumption we have that
$X \ni \overline{i}$; moreover  $\overline{i} \in V$, since   $m_i$ contains $\overline{i}$ and is contained in $V$; so we get that $V \cap X \neq \emptyset$, which is absurd.
Analogously, we can show that    $V \cap (lX) = \emptyset $ and then we can conclude.

$\bullet$ $V \cap X \neq  \emptyset $, $V \not\ni i,l$. 

In this case, we have obviously that  $V \cap (iX) \neq \emptyset $, $V \cap (lX) \neq \emptyset $, $V \not \supset (iX)  $, $V \not \supset (lX)  $ and we can conclude.

$\bullet$ $V \cap X \neq  \emptyset $,  $V \ni i,l$.
 
In this case,  we have obviously that  $V \cap (iX) \neq \emptyset $ and
 $V \cap (lX) \neq \emptyset $. Furthermore,  $V  \supset (iX) $ if and only if $ V  \supset X$ and this holds if and only if   $V \supset (lX)  $, so  we can conclude.

$\bullet$ $V \cap X \neq  \emptyset $,  $V \ni i$,   $V \not \ni l$. 

In this case,  we have obviously that $i \in \cup_{H \in {\cal H}} H$; moreover
  $V \cap (iX) \neq \emptyset $ and
 $V \cap (lX) \neq \emptyset $. Furthermore,  $V  \not \supset (lX) $ since $V \not \ni l$.
So we have to prove that  $V  \not \supset (iX) $. Suppose on the contrary that  
$V  \supset (iX) $; thus $V  \supset X $. 

If $l \not \in \cup_{H \in {\cal H}} H$, then,
by assumption, $X \ni  \hat{i}$; since $V \supset X$, we have that $V \ni \hat{i}$,
and thus $\hat{i} \in M_i$,
which is absurd. We can argue analogously 
in case $l  \in \cup_{H \in {\cal H}} H$, $M_i \cap M_l \neq \emptyset$, 
and $m_i \subset m_l$  or  $m_i \cap m_l=\emptyset$.\\
If $l  \in \cup_{H \in {\cal H}} H$, $M_i \cap M_l \neq \emptyset$, 
and $m_l \subset m_i$, then,
by assumption, $X \ni \hat{l}$; since $V \supset X$ and $M_i \supset V$ (because $V$ contains $i$), we have that   $M_i \ni   \hat{l}$; furthermore
observe that $M_i=M_l$, because if two $\cal{H}$-clusters have a nonempty intersection and are maximal, then they are equal; so $M_l \ni   \hat{l}$, which is absurd.\\
If $l  \in \cup_{H \in {\cal H}} H$ and $M_i \cap M_l = \emptyset$, then by assumption $X \ni \overline{l}$; since $X \subset V$, we have that $\overline{l}\in V$; since $V \subset M_i \ $ (because $V$ contains $i$), we get that $ \overline{l} \in M_i$ and thus  $
\overline{l} \in M_i \cap M_l $, which is absurd. 

$\bullet$ $V \cap X \neq  \emptyset $,  $V \ni l$,   $V \not \ni i$. 

Analogous to the previous case.
\end{proof}

\begin{lem} \label{aa'XX'}
Let $k, n \in \N$ with $ 4 \leq k \leq n-2$. Let ${\cal H}$ be a hierarchy on $[n]$ such that its clusters have cardinality less than or equal to $n-k$ and greater than or equal to $2$.
Let $a,a' \in [n], $ $ J \in {\cal H}$ with $a \in J $, $a' \not \in J$.
Let denote the maximal cluster containing $a'$ and the minimal cluster containing
$a'$ respectively by $M_{a'}$ and $m_{a'}$.
Let $X,X'\in {[n]-\{a,a'\} \choose k-1}$ satisfy the following conditions:

\begin{enumerate}

\item if $a' \in \cup_{H \in {\cal H}}H$, then 
\vspace*{-0.2cm}
\begin{enumerate}\itemsep0.1cm  
\item[1.1] $X$ and $X'$ 
contain an element $b$ of  $m_{a'}$ 
with  $ b \neq a'$;

\item[1.2] $X$ 
contains an element $c$ which is not in $M_{a'}$ 
and $X'$ contains an element $c'$ which is not in  $M_{a'}$; 
\end{enumerate}

\item if $a' \not\in \cup_{H \in {\cal H}}H$, then $X$ and $X'$ 
contain an element $d$ which is not in the maximal cluster containing $J$;

\item if there exists $\overline{J} $ in ${\cal H}$ with 
$a \in \overline{J} \subsetneq J $, suppose that $\overline{J}$ is maximal  among the ${\cal H}$-clusters with these
characteristics; then 
$X'$ contains an element of $J -\overline{J}$ and  $X' \cap \overline{J}= \emptyset $; 
 if there does not  exist $\overline{J} $ in ${\cal H}$ with 
$a \in \overline{J} \subsetneq J $, then  $X' \cap J \neq  \emptyset $; 

\item $X \cap J= \emptyset $; moreover, if there exists $\tilde{J}$ in ${\cal H}$ with 
$ J \subsetneq  \tilde{J}$, suppose that $\tilde{J}$ is minimal among the ${\cal H}$-clusters with these characteristics; then 
$X$ contains an element of $\tilde{J} -J$;  
\end{enumerate}

Then, in the free $\Z$-module $\oplus_{H \in {\cal H}} \Z H$, 
\begin{equation} \label{tesi} J= 
 \sum_{\stackrel{\mbox{\scriptsize $H \in {\cal H}$,}}{ \;H \cap (aX) \neq \emptyset, \; H \not \supset (aX)}}
H \;- 
\sum_{\stackrel{\mbox{\scriptsize $H \in {\cal H}$,}}{ \;H \cap (a'X) \neq \emptyset, \; H \not \supset (a'X)}}
H  \;
- 
\sum_{\stackrel{\mbox{\scriptsize $H \in {\cal H}$,}}{ \;H \cap (aX') \neq \emptyset, \; H \not \supset (aX')}} H  \; +
\sum_{\stackrel{\mbox{\scriptsize $H \in {\cal H}$,}}{ \;H \cap (a'X') \neq \emptyset, \; H \not \supset (a'X')}}
H  \end{equation}

%If $J \in {\cal H}$ and we choose $a \in J$ 
%and $a' $ not in the maximal cluster containing $J$, then,  there exist  $X,X'
%\in {[n]-\{a,a'\} \choose k-1}$  such  that 
%(\ref{tesi}) holds also if $k=3$. 

\end{lem}

\begin{proof} 
In order to prove (\ref{tesi}),
we have to show that every ${\cal H}$-cluster $V$ different from $J$
does not appear in  the second member of (\ref{tesi}) and that $J$ appears 
with coefficient $1$. Let $V \in {\cal H}$.

$\bullet$ Suppose $V \not\ni a, a'$ (so $V \neq J$). 
 
In this case $V$ does not contain any of  $aX, a'X, aX', a'X'$
and we can conclude easily by considering the four possible cases:

- $V \cap X \neq \emptyset$, $V \cap X' \neq \emptyset$, 

- $V \cap X = \emptyset$, $V \cap X' \neq \emptyset$, 

- $V \cap X \neq \emptyset$, $V \cap X' = \emptyset$, 

- $V \cap X = \emptyset$, $V \cap X' = \emptyset$.

%$\bullet$ Suppose $V \ni a, a'$ (so $V \neq J$). 
 
%In this case $V$ has nonempty intersection with each of  $aX, a'X, aX', a'X'$
%and we can conclude easily by considering the four possible cases:

%- $V \supset X,X'$ 

%- $V \not\supset X,X'$, 

%- $V \supset X$,   $V \not\supset X'$,  

%- $V \not\supset X$,   $V \supset X'$.

$\bullet$ Suppose $V \ni a'$ (so $V \neq J$).  
 
Then $a' \in \cup_{H \in {\cal H}} H$, therefore, by assumption (1), $b \in X, X'$,
$c \in X$, $c' \in X'$.
Moreover $V \ni a'$, thus $V \ni b$, so $ V \cap X \neq \emptyset$ and $ V \cap X' \neq \emptyset$. 
%Furthermore $V \not \supset aX$ and $V \not \supset aX'$ since $a \not \in V$.
Since $c \in X$, $c'  \in X'$ and $c,c' \not\in  V$, we have that $ X \not \subset V$ and 
$X' \not \subset V$, so we can conclude.

$\bullet$ Suppose $V \ni a$, $V \not \ni a'$.  
 
There are at most three possible cases:  $V \subset \overline{J}$, $V \supset \tilde{J}$, $V= J$.

- If $V \subset \overline{J}$, then $V \cap X = \emptyset$ and $V \cap X' = \emptyset$ by assumptions (3) and (4),
thus $V \not \supset (a'X')$, $ V \not \supset (a'X)$,
$V \not \supset (aX')$, $ V \not \supset (aX)$,
$V \cap (a'X') = \emptyset$ and $V \cap (a'X) = \emptyset$. Moreover, 
since $V \ni a$, $V \cap (aX') \neq \emptyset$ and $V \cap (aX)  \neq \emptyset$ and we conclude easily.

- If $V \supset \tilde{J}$, then $V \cap X \neq \emptyset $ and 
$V \cap X' \neq \emptyset $ since $\tilde{J} \cap X \neq \emptyset $ 
and  $\tilde{J} \cap X' \neq \emptyset $ by assumptions (3) and (4).  

Suppose  $a' \in \cup_{H \in {\cal H}} H$.
%$a' X \not \subset V$ and $ a' X' \not \subset V$ since $a' \not\in V$.
 Then, if $V$ contained $X$, then  it would contain $b$ and thus
 $V \cap m_{a'}$ would be nonempty; thus either $m_{a'} \subset V$ or $V \subset m_{a'}$; if $m_{a'} \subset V $, we would have $a' \in V $, which is absurd; if $ 
 V \subset m_{a'}$, we would have $c \in X \subset V \subset m_{a'} \subset M_{a'}$, so 
 $c \in M_{a'}$, which is absurd; so $V$ does not contain $X$. Analogously  
 $V$ does not contain $X'$.  So  $V \not \supset (a'X')$, $ V \not \supset (a'X)$,
$V \not \supset (aX')$, $ V \not \supset (aX)$, and we conclude.

Suppose  $a' \not\in \cup_{H \in {\cal H}} H$. Hence $X$ and $X'$ contain $d$ by assumption (2). Then, if $V$ contained $X$, then  it would contain $d$, which is absurd since $d$ is not in the maximal cluster containing $J$; 
thus $V$ does not contain $X$.  Analogously  
 $V$ does not contain $X'$.  So  $V \not \supset (a'X')$, $ V \not \supset (a'X)$,
$V \not \supset (aX')$, $ V \not \supset (aX)$, and we conclude.

- Finally consider the cluster $J$. 
We have that $J \cap X' \neq \emptyset $ by assumption (3) and $J \ni a$, so 
$J \cap (aX) \neq \emptyset $, $J \cap (aX') \neq \emptyset $, $J \cap (a'X') \neq \emptyset $.
Since $a' \not \in J $ and $J \cap X = \emptyset $ by assumption (4), we have that $J \cap (a'X) = \emptyset $. 
Moreover $J \not \supset (aX)$, since $J \cap X = \emptyset$, and 
 $J \not \supset (a'X)$ and  $J \not \supset (a'X')$, since $J \not \ni a'$.
Finally  $J \not \supset X'$, in fact: 
 if  $a' \in \cup_{H \in {\cal H}} H$, then $b \in X'$ by assumption (1), 
 so, if  $J $ contained $X'$, 
it would contain $b$, thus $J \cap m_{a'}$ would be nonempty, hence either 
$ J \subset m_{a'}$ or $ m_{a'} \subset J$; 
if $m_{a'} \subset J$, we would have $ a' \in J$, which is absurd;
if $J \subset m_{a'}$, we would have 
$c' \in X' \subset J \subset m_{a'} \subset M_{a'}$, thus $c' \in M_{a'}$, which is absurd; 
  if $a' \not \in \cup_{H \in {\cal H}} H$, then $d \in X'$ by assumption (2), so, if  $J $ contained $X'$, it would contain $d$, which is absurd.
So $J \not \supset X'$, thus 
$ J \not \supset (aX')$ and  we can conclude.

%So we have proved (\ref{tesi}).
%\smallskip

%To prove the last statement, observe that, if  
 %$a' $ not in the maximal cluster containing $J$, then, if there exists $\overline{J}\in {\cal H}$  with $a %\in \overline{J} \subsetneq J $, 
%and $X$ and $X'$ satisfy 
 % conditions 1.1, 2 and 3, then we have also that $X'$ contains an element which is not in the maximal %cluster containing $a'$. Moreover, if there exists $\tilde{J} \in {\cal H}$ with  $ J \subsetneq  \tilde{J}$,
%then we have also that $X$ contains an element which is not in the maximal cluster containing $a'$.
\end{proof}

\bigskip

{\bf Address of both authors:}
Dipartimento di Matematica e Informatica ``U. Dini'', 
viale Morgagni 67/A,
50134  Firenze, Italia

{\bf
E-mail addresses:}
baldisser@math.unifi.it, rubei@math.unifi.it


{\small \begin{thebibliography}{Dilloo Dilloo 83}

\bibitem{B-R} A. Baldisserri, E. Rubei {\sl On graphlike $k$-dissimilarity vectors}, Ann. Comb., 18 (3) 356-381 (2014)


\bibitem{B-R2} A. Baldisserri, E. Rubei {\sl 
Families of multiweights and pseudostars}, arXiv

\bibitem{B-S} H-J Bandelt, M.A. Steel {\em Symmetric matrices 
representable by weighted trees over a cancellative abelian monoid}.
SIAM J. Discrete Math. 8 (1995), no. 4, 517--525

%\bibitem{BC} C. Bocci, F. Cools {\em A tropical interpretation of
% m-dissimilarity maps} Appl. Math. Comput. 212 (2009), no. 2, 349-356 

\bibitem{B} P. Buneman {\em A note on the metric properties
 of trees}, Journal of  Combinatorial Theory Ser.  B  17 (1974), 48-50

%\bibitem{Dresslibro}  A. Dress, K. T. Huber, J. Koolen, V. Moulton, 
%A. Spillner, Basic phylogenetic combinatorics. Cambridge University %Press, Cambridge, 2012

\bibitem{H-H-M-S} 
S.Herrmann, K.Huber, V.Moulton, A.Spillner, 
{\em Recognizing treelike k-dissimilarities},
J. Classification 29  (2012),   no. 3, 321-340


\bibitem{H-Y}
S.L. Hakimi, S.S. Yau,
{\em Distance matrix of a graph and its realizability},
Quart. Appl. Math. 22 (1965), 305-317

%\bibitem{Iri} B. Iriarte Giraldo
%{\em Dissimilarity vectors of trees are contained in the 
%tropical Grassmannian},  Electron.  J.  Combin. 17 (2010), no. 1

\bibitem{L-Y-P} D. Levy, R. Yoshida, L. Pachter {\em Beyond pairwise
distances: neighbor-joining with phylogenetic diversity esitimates},
Mol. Biol. Evol.  23  (2006), no. 3,  491-498 


%\bibitem{Man} C. Manon, {\em Dissimilarity maps on trees and the %representation theory of $SL_m( \C)$}, J. Algebraic Combin. 33 (2011), %no. 2, 199-213


\bibitem{P-S}  L. Pachter, D. Speyer {\em Reconstructing
 trees from subtree weights}, Appl. Math. Lett. 17  (2004), no. 6, 615--621

\bibitem{Ru1}
E. Rubei {\em Sets of double and
 triple weights of trees}, Ann. Comb. 15  (2011), no. 4, 723-734  

\bibitem{Ru2} E. Rubei {\em On dissimilarity vectors of general 
weighted trees}, Discrete Math. 312  (2012), no. 19,  2872-2880 

\bibitem{SimP} J.M.S. Simoes Pereira {\em A Note on the Tree 
Realizability of a distance matrix}, J. Combinatorial Theory 6 (1969),
 303-310

%\bibitem{SS2}
%D. Speyer, B. Sturmfels {\em Tropical mathematics}, 
%Math. Mag. 82  (2009), no. 3, 163-173




\bibitem{Za} K.A. Zaretskii {\em Constructing trees from the set of distances between pendant vertices}, Uspehi Matematiceskih Nauk. 20 (1965), 90-92 

\end{thebibliography}}
\end{document}